\documentclass[a4paper,11pt]{amsart}

\usepackage[top=3.5cm,bottom=3.5cm,left=3.6cm,right=3.6cm]{geometry}
\usepackage{amsfonts}
\usepackage{amsmath}
\usepackage{amssymb}
\usepackage{amsthm}
\usepackage[T1]{fontenc}
\usepackage{graphicx}
\usepackage{hyperref}
\usepackage{enumerate}
\usepackage{varioref}
\usepackage{enumitem}

\newtheorem{dfn}{Définition}
\newtheorem{propo}[dfn]{Proposition}

\newtheorem*{theo}{Theorem}
\newtheorem{innercustomthm}{Theorem}
\newenvironment{customthm}[1]
  {\renewcommand\theinnercustomthm{#1}\innercustomthm}
  {\endinnercustomthm}

\newtheorem{cor}[dfn]{Corollary}
\newtheorem{lem}[dfn]{Lemma}
\theoremstyle{remark}

\newtheorem{rem}[dfn]{Remark}

\newcommand{\suite}[2]{{\left(#1\right)}_{#2}}
\newcommand{\norm}[1]{\left\lVert #1 \right\rVert}
\newcommand{\ens}[1]{\left\{ #1\right\}}
\newcommand{\R}{\mathbb{R}}
\newcommand{\Z}{\mathbb{Z}}

\newcommand{\ent}[1]{\lfloor #1\rfloor}

\newcommand{\f}{\mathcal{F}}

\newcommand{\el}{\mathbb{L}}

\newcommand{\ds}{\displaystyle}

\newcommand{\abs}[1]{\left|#1\right|}

\providecommand{\newoperator}[3]{%
  \newcommand*{#1}{\mathop{#2}#3}}
\newoperator{\re}{\mathrm{Re}}{\,}
\newoperator{\im}{\mathrm{Im}}{\,}
\newcommand{\croch}[1]{\left[#1\right]}

\begin{document}

\title{A strictly stationary
$\beta$-mixing process satisfying the central limit theorem but not the weak 
invariance principle}

\author{Davide Giraudo \and Dalibor Voln\'y}

\address{Universit\'e de Rouen, LMRS, Avenue de l'Universit\'e, BP 12 76801
Saint-\'Etienne-du-Rouvray cedex, France.}

\email{davide.giraudo1@univ-rouen.fr\and dalibor.volny@univ-rouen.fr}

\date{\today}



 \begin{abstract}
 In 1983, N. Herrndorf proved that for a $\phi$-mixing sequence satisfying the 
 central limit theorem and 
$\liminf_{n\to\infty}\sigma^2_n/n>0$,
the weak invariance principle takes place. The question whether for strictly
stationary sequences with finite second 
moments and a weaker type
($\alpha$, $\beta$, $\rho$) of mixing the central limit theorem implies the
weak invariance principle remained open. 

We construct a strictly stationary $\beta$-mixing sequence with finite moments 
of any order and linear variance for which the central limit theorem takes
place but not the weak invariance principle.
 \end{abstract}
 \maketitle

\section{Introduction and notations}
Let $(\Omega, \f,\mu)$ be a probability space. If $T\colon \Omega\to \Omega$ 
is one-to-one, bi-measurable and 
measure preserving (in sense that $\mu(T^{-1}(A))=\mu(A)$ for all $A\in\f$),
then the sequence $\suite{f\circ T^k}{k\in\Z}$ is strictly stationary for any 
measurable $f\colon \Omega\to\R$. Conversely, each strictly stationary 
sequence can be represented in this way.

For a zero mean square integrable $f\colon \Omega\to \R$, we define $\ds 
S_n(f):=\sum_{j=0}^{n-1}f\circ T^j$, 
$\sigma_n^2(f):=\mathbb E(S_n(f)^2)$ and 
$S_n^*(f,t):=
S_{\ent{nt}}(f)+(nt-\ent{nt})f\circ T^{\ent{nt}}$, where $\ent x$ is the 
greatest integer which is 
less than or equal to $x$.

We say that $(f\circ T^j)_{j\geqslant 1}$ satisfies the \textit{central limit 
theorem with normalization $a_n$} if 
the sequence $(a_n^{-1}S_n(f))_{n\geqslant 1}$ converges weakly to a strandard
normal distribution. 
Let $C[0,1]$ denote the space of continuous functions on the unit interval 
endowed with the norm 
$\ds\norm{g}_\infty:=\sup_{t\in [0,1]}\abs{g(t)}$.

Let $D[0,1]$ be the space of real valued functions which have left limits and 
are continuous-from-the-right at each 
point of $[0,1)$. We endow it with Skorohod metric (cf. \cite{MR0233396}). We define $S_n^{**}(f,t)
:=S_{\ent{nt}}(f)$, which gives a random element of $D[0,1]$.

We shall say that the strictly stationary sequence $\suite{f\circ T^j}{j\geqslant 0}$ \textit{satisfies the weak 
invariance principle in $C[0,1]$ with normalization $a_n$} 
(respectively \textit{in} $D[0,1]$) if the sequence of $C[0,1]$ (of $D[0,1]$) valued
random variables $\suite{a_n^{-1}S_n^*(f,\cdot)}{n\geqslant 1}$ (resp. 
$\suite{a_n^{-1}S_n^{**}(f,\cdot)}{n\geqslant 1}$) weakly converges to a Brownian motion process in the 
corresponding space.

Let $\mathcal A$ and $\mathcal B$ be two sub-$\sigma$-algebras of $\mathcal F$, where $(\Omega,\mathcal F,\mu)$ is a 
probability space. We define the $\alpha$-mixing coefficients as introduced by Rosenblatt in \cite{MR0074711}:
\begin{equation*}
\alpha(\mathcal A,\mathcal B):=\sup\ens{\abs{\mu(A\cap B)-\mu(A)\mu(B)},A\in 
\mathcal A,B\in \mathcal B}.
\end{equation*}
Define the $\beta$-mixing coefficients by
\begin{equation*}
\beta(\mathcal A,\mathcal B):=\frac 12\sup\sum_{i=1}^I\sum_{j=1}^J\abs{\mu(A_i\cap B_j)-\mu(A_i)\mu(B_j)},
\end{equation*}
where the supremum is taken over the finite partitions $\ens{A_i,1\leqslant i\leqslant I}$ and $\ens{B_j,1\leqslant 
j\leqslant J}$ of $\Omega$ of elements of $\mathcal A$ (respectively of $\mathcal B$). They were introduced by 
Volkonskii and Rozanov \cite{MR0105741}.

The $\rho$-mixing coefficients were introduced by Hirschfeld \cite{PSP:1737020} and are defined by
\begin{equation*}
 \rho(\mathcal A,\mathcal B):=\sup\ens{\abs{\operatorname{Corr}(f,g)},f\in \mathbb L^2(\mathcal A),g\in \mathbb 
L^2(\mathcal B)},
\end{equation*}
where $\operatorname{Corr}(f,g):=\left[\mathbb E(fg)-\mathbb E(f)\mathbb E(g)\right]\left[\norm{f-\mathbb E(f)}_{\mathbb L^2}
\norm{g-\mathbb E(g)}_{\mathbb L^2}\right]^{-1}$. 

Ibragimov \cite{MR0105180}
introduced for the first time $\phi$-mixing coefficients , which are given by the formula
\begin{equation*}
 \phi(\mathcal A,\mathcal B):=\sup\ens{\abs{\mu(B\mid A)-\mu(B)},A\in\mathcal A,B\in\mathcal B,\mu(A)>0}.
\end{equation*}

The coefficients are related by the inequalities
\begin{equation}\label{comp_mix}
  2\alpha(\mathcal A,\mathcal B)\leqslant\beta(\mathcal A,\mathcal B)\leqslant \phi(\mathcal A,\mathcal B),\quad
\alpha(\mathcal A,\mathcal B)\leqslant \rho(\mathcal A,\mathcal B)\leqslant 2\sqrt{\phi(\mathcal A,\mathcal B)}.
\end{equation}

For a strictly stationary sequence $\suite{X_k}{k\in\Z}$ and $n\geqslant 0$ we define $\alpha_X(n) = \alpha(n) =
\alpha(\mathcal F_{-\infty}^0,\mathcal F_n^\infty)$
where $\mathcal F_u^v$ is the $\sigma$-algebra generated by $X_k$ with $u\leqslant k\leqslant v$ (if 
$u=-\infty$ or $v=\infty$, the corresponding
inequality is strict). In the same way we define coefficients $\beta_X(n)$, $\rho_X(n)$, $\phi_X(n)$.

We say that the sequence $\suite{X_k}{k\in\Z}$ is $\alpha$\textit{-mixing}
 if $\ds\lim_{n\to +\infty}\alpha_X(n)=0$, and similarily we define $\beta$, $\rho$ and $\phi$-mixing sequences. 

$\alpha$, $\beta$ and $\phi$-mixing sequences were considered in the mentioned references, 
while $\rho$-mixing sequences first appeared in \cite{MR0133175}.
 
 Inequalities \eqref{comp_mix} give a hierarchy between theses classes of mixing sequences. 

If $(a_N)_{N\geqslant 1}$ and $(b_N)_{N\geqslant 1}$ are two sequences of positive real numbers, we write 
$a_N \asymp b_N$ if there exists a positive constant $C$ such that for each $N$, $C^{-1}a_N\leqslant b_N\leqslant 
Ca_N$. 

The main results are 

\begin{customthm}{A}\label{A}
Let $\delta$ be a positive real number. There exists
a strictly stationary real valued process $Y=\suite{Y_k}{k\geqslant 0}=\suite{f\circ T^k}{k\geqslant 0}$ 
satisfying the following conditions:
\begin{enumerate}[label=\alph*)]
 \item\label{clt} the central limit theorem with normalization $\sqrt n$ takes place;
 \item\label{nwip} the weak invariance principle with normalization $\sqrt n$ does not hold;
 \item \label{SD} we have $\sigma_N(f)^2\asymp N$;
 \item\label{control_mixing} we have for some positive $C$, $\beta_Y(N)\leqslant C/N^{1/2-\delta}$;
 \item\label{moments} $Y_0\in\el^p$ for any $p>0$.
\end{enumerate}
\end{customthm}

Alternatively, we can construct the process in order to have a control of the mixing coefficients on a 
subsequence. 

\begin{customthm}{A'}\label{A'}
Let $(c_j)_{j\geqslant 0}$ be a decreasing sequence of positive numbers. Then there exists
a strictly stationary real valued process $Y=\suite{Y_k}{k\geqslant 0}=\suite{f\circ T^k}{k\geqslant 0}$ 
satisfying conditions \ref{clt}, \ref{nwip}, \ref{SD}, \ref{moments} in Theorem~\ref{A}, and:
\begin{enumerate}
 \item[d')] \label{control_mixing2}
there is an increasing sequence $(m_k)_{k\geqslant 1}$ of integers such that for each 
 $k$, $\beta_Y(m_k)\leqslant c_{m_k}$.
\end{enumerate}
\end{customthm}

\begin{rem}
 Herrndorf proved (\cite{MR699789}, Theorem 2.13) that if $(\xi_n)$ is a strictly stationary $\phi$-mixing 
 sequence for which $\sigma_n\to\infty$, $S_n/\sigma_n$ converges in distribution to a standard normal 
 distribution and $\sigma_n^{-1}\max_{1\leqslant i\leqslant n}\abs{\xi_i}\to 0$ in probability, then the weak 
 invariance principle takes place. So Herrndorf's result does not extend to $\beta$-mixing sequences.

\end{rem}

\begin{rem}
Rio \textit{et al.} proved in \cite{MR2117923} that the condition 
$\int_0^1\alpha^{-1}(u)Q^2(u)du<\infty$ 
implies the weak invariance principle, where $\alpha^{-1}(u):=\inf\ens{k,
\alpha(k)\leqslant u}$ and $Q$ is the 
right-continuous inverse of the quantile function $t\mapsto \mu\ens{X_0> t}$. 
If the process is strictly stationary, with finite moments of order $2+r$, 
$r>0$, the latter condition is 
satisfied whenever $\sum_{n=1}^\infty(n+1)^{2/r}\alpha(n) < \infty$ 
(Ibragimov \cite{MR0148125} found the 
condition $\sum_{n=1}^\infty\alpha(n)^{1-2/r} < \infty$). Since 
$Y_0\in\mathbb \el^p$ for all 
$p<\infty$, we have 
that $\sum_N\alpha(N)^r=+\infty$ for any $r<1$, hence 
in Theorem~\ref{control_mixing2} we can thus hardly get such a bound as in 
d') for the whole sequence.
\end{rem}

\begin{rem}
 Ibragimov proved that for a strictly stationary $\rho$-mixing sequence with finite moments of order $2+\delta$ for 
some positive $\delta$, the weak invariance principle holds, cf.~\cite{MR0362448}. In particular, this proves that 
our construction does not give a $\rho$-mixing process. Shao also showed in \cite{MR1038376} that the condition 
$\sum_n\rho(2^n)<\infty$ is sufficient in order to guarantee the weak invariance principle in $D[0,1]$ for 
stationary sequences having order two moments. So a potential $\rho$-mixing counter-example has to adhere to   
restrictions on the moments as well as on the mixing rates. 
\end{rem}

\medskip

\centerline{\textbf{About the method of proof}}
\smallskip

In proving the result we will use properties of coboundaries $h=g - g\circ T$ ($g$ is called a transfer 
function). 
For a positive integer $N$ and a measurable function $v$, we denote
$\ds S_N(v):=\sum_{j=0}^{N-1}U^jv$ (Here and below, $U^jv:=v\circ T^j$.).
Because $S_n(g - g\circ T)
= g-g\circ T^n$, for any sequence $a_n\to\infty$ we have $(a_n)^{-1}S_n(g - g\circ T) \to 0$ in probability hence 
adding a coboundary does not
change validity of the central limit theorem. If, moreover, $g\in\mathbb L^2$ then $n^{-1/2}\norm{S_n^*(g - 
g\circ T)}_{\infty} \to 0$ a.s. hence adding of such 
coboundary does not change validity of the invariance principle (if norming by $\sqrt n$ or by $\sigma_n$ with 
$\liminf_n \sigma^2_n/n>0$), cf. \cite{MR624435}, pages 140-141. On 
the other hand, if $g\not\in\mathbb L^2$, adding a coboundary can spoil tightness even if $g - g\circ T$ is 
square integrable, cf. \cite{MR1893125}. A similar idea 
was used in \cite{MR2359065}. In the proof of Theorem~\ref{A} and \ref{A'} we will find a coboundary 
$g - g\circ T$ which 
is $\beta$-mixing and spoils tightness. The coboundary
has all finite moments but the transfer function is not integrable. We then add an $m$ such that $(m\circ 
T^i)_{i\in\Z}$ and $(h\circ T^i)_{i\in\Z}$ are independent (enlarging the probability 
space), and $m\circ T^i$ is i.i.d. with moments of any order (in particular, it satisfies the weak invariance 
principle).

The proof uses the fact that 
$\abs{\mu(A\cap B)-\mu(A)\mu(B)}\leqslant \mu(A)$. The method does not seem to apply to processes which are 
$\rho$-mixing and for this kind of processes the problem remains open.

\section{Proof}

\subsection{Construction of $h$}

Let us consider an increasing sequence of positive integers 
$\suite{n_k}{k\geqslant 1}$ such that
\[n_1\geqslant 2\quad\mbox{and}\quad\sum_{k=1}^{\infty}\frac 1{n_k}<\infty,\]
 and for each integer $k\geqslant 1$, let $A^{-}_k, A^{+}_k$
be disjoint measurable sets such that $\mu(A^{-}_k)= 1/(2n_k^2)=\mu(A^{+}_k)$.

Let the random variables $e_k$ be defined by
\begin{equation}
 e_k(\omega):=\begin{cases}
               1&\mbox{ if }\omega\in A^{+}_k,\\
		-1&\mbox{ if }\omega\in A^{-}_k,\\
		0&\mbox{ otherwise}.	
              \end{cases}
\end{equation}
We can choose the dynamical system $(\Omega,\mathcal F,\mu,T)$ and the sets 
$A_k^+,A_k^-$ in such a 
way that the family $(e_k\circ T^i)_{k\geqslant 1,i\in\Z}$ is independent. 
We define 
$A_k:=A_k^{+}\cup A_k^{-}$ and
\begin{equation}
 h_k:=\sum_{i=0}^{n_k-1}U^{-i}e_k-U^{-n_k}\sum_{i=0}^{n_k-1}U^{-i}e_k, 
 \quad \mbox{} h:=\sum_{k=1}^{+\infty}h_k.
\end{equation}

Since $\mu\ens{h_k\neq 0}\leqslant 2/n_k$, the function $h$ is almost everywhere well-defined (by 
the Borel-Cantelli lemma).

It will be useful to express, for $N\geqslant n_k$, the sum $S_N(h_k)$ as a linear combination of 
$U^pe_k$. Denote $\ds s_k:=\sum_{j=0}^{n_k-1}U^{-j}e_k$. As $N\geqslant n_k$ and $h_k=s_k-U^{-n_k}s_k$, we have
\begin{align*}
 S_N(h_k)&= \sum_{j=0}^{N-1}(U^js_k-U^{j-n_k}s_k)\\
&=\sum_{j=0}^{N-1}U^js_k-\sum_{j=-n_k}^{N-n_k-1}U^js_k \\
&=-\sum_{j=-n_k}^{-1}U^js_k+U^N\sum_{j=-n_k}^{-1}U^js_k.
\end{align*}
We also have
\begin{align}
\sum_{j=-n_k}^{-1}U^js_k&=\sum_{j=1}^{n_k}U^{-j}s_k\nonumber\\
&=U^{-1}\sum_{i=0}^{n_k-1}\sum_{j=0}^{n_k-1}U^{-(i+j)}e_k\nonumber\\
\label{expr_sum2} \sum_{j=-n_k}^{-1}U^js_k 
&=U^{-2n_k}\left(\sum_{j=1}^{n_k}jU^{j}e_k+\sum_{j=1}^{n_k-1}(n_k-j)U^{n_k+j}e_k\right).
\end{align}
The previous equations yield
\begin{multline}\label{powers}
S_N(h_k)=\sum_{j=1}^{n_k}jU^{j+N-2n_k}e_k+
\sum_{j=1}^{n_k-1}(n_k-j)U^{j+N-n_k}e_k\\\
-\sum_{j=1}^{n_k}jU^{j-2n_k}e_k-
\sum_{j=1}^{n_k-1}(n_k-j)U^{j-n_k}e_k.
\end{multline}
Each $h_k$ is a coboundary, as if we define $\ds v_k:=\sum_{j=0}^{n_k-1}U^{-j}s_k$, then
$v_k-U^{-1}v_k=s_k-U^{-n_k}s_k=h_k$ (so in this case the transfer function is $-U^{-1}v_k$).

Since $\mu\ens{v_k\neq 0}\leqslant 2/n_k$,
Borel-Cantelli's lemma shows that the function $\ds g:=-\sum_{k=1}^{+\infty}U^{-1}v_k$ is almost everywhere 
well defined under our assumption that $\sum_k 1/n_k$ is convergent. Because $h=g-Ug$, $h$ is a coboundary.

\subsection{Mixing rates}

We show that the process $(U^if)_{i\in\Z}$ is $\beta$-mixing. In doing so we use the following proposition (cf. 
\cite{MR2325294}, Theorem~6.2).

\begin{propo}\label{indep_mixing}
Let $(X_{k,i})_i$, $k=1,2,\dots$ be mutually independent strictly stationary processes with respective mixing 
coefficients $\beta_k(n)$,
let $X_i = \sum_{k=1}^\infty X_{k,i}$ converge. The process $(X_i)_i$ is strictly stationary with mixing 
coefficients
$\beta(n) \leqslant \sum_{k=1}^\infty \beta_k(n)$.
\end{propo}

This reduces the proof of $\beta$-mixing of $(U^if)_{i\in\Z}$ (in 
Theorems~\ref{A} and~\ref{A'}) to that of $(U^ih)_{i\in\Z}$ and thereby to that 
of $(U^ih_k)_{i\in\Z}$ for $k\geqslant 1$.

In the following text we denote by $\beta_k(n)$ the mixing coefficients of 
the process $(h_k\circ T^i)_{i\in\Z}$.

\begin{lem}\label{mix}
For $k\geqslant 1$, we have the estimate $\beta_k(0) \leqslant 4/n_k$.
\end{lem}

\begin{proof}
 Suppose $k$ is a positive integer. For $-\infty\leqslant j\leqslant l\leqslant \infty$, let $\mathcal H_j^l$ 
 denote the $\sigma$-field generated by $U^ih_k,j\leqslant i\leqslant l$, ($i\in\Z$), and let $\mathcal G_j^l$ 
 denote the $\sigma$-field generated by $U^ie_k$, $j\leqslant i\leqslant l$, ($i\in\Z$). 
 Define the $\sigma$-fields $\mathcal B_1:=\mathcal G_{-\infty}^{-2n_k}$, $\mathcal B_2:=\mathcal G_{-2n_k+1}^0$ and 
 $\mathcal B_3:=\mathcal G_1^\infty$. Now $\mathcal H_{-\infty}^0\subset \mathcal B_1\vee \mathcal B_2$ 
 and $\mathcal H_0^\infty\subset\mathcal B_2\vee\mathcal B_3$. Therefore $\beta_k(0)\leqslant 
 \beta(\mathcal B_1\vee \mathcal B_2,\mathcal B_2\vee \mathcal B_3)$. The $\sigma$-fields $\mathcal B_1$, 
 $\mathcal B_2$, $\mathcal B_3$, are independent; hence the $\sigma$-fields $\mathcal B_1\vee \mathcal B_3$ and 
 $\mathcal B_2\vee \mathcal B_2$ (with index $2$ in both places) are 
 independent; this implies by a result given 
 e.g. in \cite[Theorem~6.2]{MR2325294},
 \[\beta(\mathcal B_1\vee \mathcal B_2,\mathcal B_2\vee \mathcal B_3)\leqslant 
 \beta(\mathcal B_1, \mathcal 
 B_3)+\beta(\mathcal B_2, \mathcal B_2)= 0+\beta(\mathcal B_2,\mathcal B_2).\]
 Thus $\beta_k(0)\leqslant \beta(\mathcal B_2,\mathcal B_2)$. Also, the $\sigma$-field $\mathcal B_2$ 
 has an atom $P_0:=\bigcap_{i=-2n_k+1}^0\ens{U^ie_k=0}$ that satisfies $\mu(P_0)\geqslant 1-2/n_k$ 
 (since 
 $\mu(U^ie_k\neq 0)=1/n_k^2$ for each $i$). By Lemma 2.2 of \cite{MR806228}, if 
 $\mathcal B$ is a $\sigma$-field which has an atom $D$, then 
 $\beta(\mathcal B,\mathcal B)
 \leqslant 2[1-\mu(D)]$. Hence $\beta(\mathcal B_2,\mathcal B_2)\leqslant 2[1-\mu(P_0)]\leqslant 4/n_k$. 
 This finishes the proof of Lemma~5.
\end{proof}

Denoting by $\beta(N)$ the mixing coefficients of the sequence $(h\circ T^i)_{i\in \Z}$, 
Proposition~\ref{indep_mixing}, Lemma~\ref{mix} and the fact that $\beta_k(N)=0$ when $N\geqslant 
2n_k$ yield

\begin{cor}
 For each integer $k$, we have 
\begin{equation}\label{mixing_estimate}
 \beta(N)\leqslant \sum_{j\geqslant 1}\beta_j(N)\leqslant \sum_{j:2n_j\geqslant N}\frac 4{n_j}.
\end{equation}

\end{cor}

Now we can prove \ref{control_mixing} and d').
Let $i(N)$ denote the unique integer such that 
$n_{i(N)}\leqslant N< n_{i(N)+1}$. For sequences 
$(u_N)_{N\geqslant 1}$, $(v_N)_{N\geqslant 1}$ of positive numbers, 
$u_N\lesssim v_N$ means that there is $C>0$ such 
that for each $N$, $u_N\leqslant C\cdot v_N$.  

\begin{propo}\label{mixing_delta}
Let $\delta>0$. 
With the choice $n_k:=\lfloor 16^{(2+\delta)^k}\rfloor$, we have \ref{control_mixing}.
\end{propo}

\begin{proof}
We deduce from \eqref{mixing_estimate}
\begin{equation*}
\beta(2n_k)\leqslant \sum_{j\geqslant k}\frac 4{n_j}\lesssim \sum_{j\geqslant 0}16^{-(2+\delta)^k(2+\delta)^j}\lesssim 16^{-(2+\delta)^k}.
\end{equation*} 
Consequently,
\begin{equation*}
\beta(2N)\leqslant \beta(2n_{i(N)})\lesssim\frac 4{n_{i(N)}}\lesssim \frac{1}{n_{i(N)+1}^{\frac 
1{2+\delta}}}\lesssim \frac{1}{N^{\frac 1{2+\delta}}}.
\end{equation*}
\end{proof}
Hence \ref{control_mixing} is fulfilled.

\begin{propo}\label{mixing_subsequence}
Given $\suite{c_k}{k\geqslant 1}$ as in Theorem~\ref{A'}, one can 
recursively choose a sequence $(n_k)
_{k\geqslant 1}$ growing to infinity
arbitrarily fast, such that for the construction given above, 
one has that for each $k\geqslant 1$, 
$\beta(2n_k)\leqslant c_{2n_k}$.

\end{propo}

 \begin{proof}  Suppose 
 that the sequence $(n_k)_{k\geqslant 1}$ satisfies  
 \[n_{k+1}\geqslant  \frac 8{c_{2n_k}}\quad \mbox{and}\quad n_{k+1}\geqslant 2n_k,\quad k\geqslant 1.\]
 Then 
 \[\beta(2n_k)\leqslant \sum_{j=k+1}^\infty\beta_j(2n_k)\leqslant \sum_{j=k+1}^\infty\beta_j(0),\]
 and, by Lemma~\ref{mix} and the condition $n_{j+l}\geqslant 2^ln_j$ for $j,l\geqslant 1$, we derive 
 \[\beta(2n_k)\leqslant \sum_{j=k+1}^\infty\frac 4{n_j}\leqslant \frac 8{n_{k+1}}.\]
 The assumption $n_{k+1}\geqslant  8/c_{2n_k}$ yields $\beta(2n_k)\leqslant c_{2n_k}$ for 
 each $k\geqslant 1$.

 \end{proof}
%
%
%

This proves d') with $m_k:=2n_k$.

\begin{rem}\label{remark_mixing}
The sequence of integers $(n_k)_{k\geqslant 1}$ can be chosen to meet all other
conditions imposed in this paper.
\end{rem}

\subsection{Proof of non-tightness}

\begin{lem}\label{max}
There exists $N_0$ such that
\begin{equation}
\mu\ens{\max_{2n_k\leqslant N\leqslant n_k^2}\abs{S_N(h_k)}\geqslant n_k}>1/4
\end{equation}
whenever $n_k\geqslant N_0$.
\end{lem}
\begin{proof}
For $2n_k\leqslant N\leqslant n_k^2$, thanks to \eqref{powers}, we have
\begin{equation*}
 \ens{\abs{S_N(h_k)}=n_k}\supset \ens{\abs{U^{N-n_k}e_k}=1}\cap \bigcap_{j\in 
 I}\ens{U^je_k=0}\cap\bigcap_{j\in 
J_N}\ens{U^je_k=0},
\end{equation*}
where $I=[1-2n_k,-1]\cap \Z$ and $J_N=([N-2n_k+1,N-1-n_k]\cup [N+1-n_k,N-1])
\cap\Z$. We define
\begin{align*}
 B_{N,k}&:=\ens{\abs{U^{N-n_k}e_k}=1}\cap \bigcap_{j=1-2n_k}^{-1-n_k}
 \ens{U^{N+j}e_k=0}\cap \bigcap_{j=1-n_k}^{-1}
\ens{U^{j+N}e_k=0}\\ &= \ens{\abs{U^{N-n_k}e_k}=1}\cap \bigcap_{j\in J_N} 
\ens{U^je_k=0}.
\end{align*}
We have $\abs{S_{N}(h_k)}=n_k$ on $\bigcap_{j\in I}\ens{U^je_k=0} \cap B_{N,k}$
and the sets 
$\bigcap_{j\in I}\ens{U^je_k=0}$, 
$\bigcup_{N=2n_k}^{n_k^2}B_{N,k}$ belong
to independent $\sigma$-algebras. Therefore
\begin{equation}\label{reduction}
 \mu\ens{\max_{2n_k\leqslant N\leqslant n_k^2}\abs{S_N(h_k)}\geqslant n_k}
 \geqslant \left(1-\frac 1{n_k^2}\right)^{2n_k}
 \mu\left(\bigcup_{N=2n_k}^{n_k^2}B_{N,k}\right).
\end{equation}
Recall Bonferroni's inequality, which states that for any integer $n$ and any 
events $A_j,j\in\ens{1,\dots,n}$, we have
\begin{equation}\label{bonf}
 \mu\left(\bigcup_{j=1}^nA_j\right)\geqslant \sum_{j=1}^n\mu(A_j)-
 \sum_{1\leqslant i<j\leqslant n}\mu(A_i\cap A_j).
\end{equation}
It can be proved by induction.
Notice that
$$\mu(B_{N,k}) = \frac1{n_k^2}\left(1 - \frac1{n_k^2}\right)^{2n_k-2} \geqslant
\frac1{n_k^2}\left(1 - \frac2{n_k}\right)$$
and for $i\neq j$
$$\mu(B_{i+2n_k-1,k}\cap B_{j+2n_k-1,k}) \leqslant  
\mu\ens{\abs{U^{i+n_k-1}e_k}=1}\mu\ens{\abs{U^{j+n_k-1}e_k}=1}= 
\frac1{n_k^4}$$
hence
\begin{align*}
 \mu\left(\bigcup_{N=2n_k}^{n_k^2}B_{N,k}\right)&=\mu\left(
 \bigcup_{N=1}^{(n_k-1)^2}B_{N+2n_k-1,k}\right)\\
&\geqslant \sum_{N=1}^{(n_k-1)^2}\frac 1{n_k^2}\left(1-\frac 2{n_k}\right)-
\sum_{1\leqslant i<j\leqslant (n_k-1)^2}
\mu(B_{i+2n_k-1,k}\cap B_{j+2n_k-1,k})\\
&\geqslant  \left(1-\frac 2{n_k}\right)^3 - \frac12
\end{align*}
which together with \eqref{reduction} and the inequality $(1-1/n_k^2)^{2n_k}
\geqslant 1-2/n_k$ implies that
\begin{equation}
 \mu\ens{\max_{2n_k\leqslant N\leqslant n_k^2}\abs{S_N(h_k)}\geqslant n_k}> \frac 14,
\end{equation}
whenever $n_k\geqslant N_0$, where $N_0\geqslant 3$ is such that
$\left(1-2/n\right)\left[(1-2/n)^3-1/2\right] > 1/4$ for $n\geqslant N_0$.
\end{proof}

\begin{lem}\label{lacunarity}
Assume that the sequence $\suite{n_k}{k\geqslant 1}$ satisfies the following two conditions of lacunarity: 
\begin{eqnarray}
\label{lac1}\mbox{ for each}\: k\geqslant K, \quad 16\sum_{j=1}^kn_j^2\leqslant n_{k+1};\\
\label{lac2}\mbox{ for each}\: k\geqslant K, \quad  n_{k+1}\geqslant (k+1)^2n_k,
\end{eqnarray}
where $2\leqslant K<\infty$.
Then we have for $k$ large enough
\begin{equation}\label{non_tight}
\mu\ens{\frac 1{n_k}\max_{2n_k\leqslant N\leqslant n_k^2}\abs{S_N(h)}\geqslant 1/2}\geqslant 1/8.
\end{equation}
\end{lem}

\begin{proof}
Fix an integer $k\geqslant K$. For $2n_k\leqslant N\leqslant n_k^2$, we have
\begin{equation}\label{upper_bound}
\abs{\frac 1{n_k}\sum_{j=1}^{k-1}S_N(h_j)}\leqslant 
\frac 2{n_k}\sum_{j=1}^{k-1}(n_j+1)^2\leqslant \frac 12.
\end{equation}
Indeed, using \eqref{powers}, for $j<k$, we can give an upper bound of 
$S_N(h_j)$ (as $N\geqslant 2n_k>2n_j$) as
\begin{equation*}
\abs{S_N(h_j)}\leqslant 2\sum_{l=1}^{n_j}l+ 2\sum_{l=1}^{n_j-1}l=
2n_j^2,
\end{equation*}
and \eqref{upper_bound} holds by \eqref{lac1}.

Now fix $j>k$. 
Writing $\ds S_N(h_j)=\sum_{i=0}^{N-1}U^is_j-\sum_{i=-n_j}^{N-n_j-1}U^is_j$, where 
$s_j:=\sum_{i=0}^{n_j-1}U^{-i}e_j$, we can see that
\[\bigcup_{N=2n_k}^{n_k^2}\ens{S_N(h_j)\neq 0}\subset 
\bigcup_{i=-2n_j+1}^{n_k^2}T^{-i}A_j,\] hence using 
\eqref{lac2}
\begin{equation}\label{support}
\mu\left(\bigcup_{N=2n_k}^{n_k^2}\ens{S_N(h_j)\neq 0}\right)\leqslant \frac{n_k^2+2n_j}{n_j^2}\leqslant
\frac{2n_k}{n_j}\leqslant 2j^{-2}.
\end{equation}
Let $E_k:=\bigcup_{N=2n_k}^{n_k^2}\bigcup_{j\geqslant k+1}\ens{S_N(h_j)\neq 0}$. 
By
\eqref{support}, we have 
\begin{equation}\label{support2}
 \mu(E_k)\leqslant \sum_{j\geqslant k+1} 2j^{-2}\leqslant 2/k.
\end{equation}

By \eqref{upper_bound},
\begin{equation*}
 \frac 1{n_k}\max_{2n_k\leqslant N\leqslant n_k^2}\abs{S_N(h)}\geqslant
\frac 1{n_k}\max_{2n_k\leqslant N\leqslant n_k^2}\abs{S_N\left(\sum_{j\geqslant k}h_j\right)}-\frac 12,
\end{equation*}
hence
\begin{align*}
\mu\ens{\frac 1{n_k}\max_{2n_k\leqslant N\leqslant n_k^2}\abs{S_N(h)}\geqslant \frac 12}
&\geqslant \mu\ens{\frac 1{n_k}\max_{2n_k\leqslant N\leqslant n_k^2}\abs{S_N\left(\sum_{j\geqslant 
k}h_j\right)}\geqslant 1}\\
&\geqslant \mu\left(\ens{\frac 1{n_k}\max_{2n_k\leqslant N\leqslant n_k^2}\abs{S_N\left(\sum_{j\geqslant 
k}h_j\right)}\geqslant 1}\cap E_k^c\right) \\
&=\mu\left(\ens{\frac 1{n_k}\max_{2n_k\leqslant N\leqslant n_k^2}\abs{S_N\left(h_k\right)}\geqslant 1}\cap 
E_k^c\right) \\
&\geqslant \mu\left(\ens{\frac 1{n_k}\max_{2n_k\leqslant N\leqslant n_k^2}\abs{S_N\left(h_k\right)}\geqslant 
1}\right)-\mu(E_k).
\end{align*}
The result follows from Lemma~\ref{max} and \eqref{support2}.
\end{proof}

 The previous lemma yields together with Theorems 8.1 and 15.1 of \cite{MR0233396} 
and the convergence to $0$ of the finite dimensional distributions of $(N^{-1/2}S_N^*(h))_{N\geqslant 1}$ and 
$(N^{-1/2}S_N^{**}(h))_{N\geqslant 1}$ the following corollary.

\begin{cor}\label{cor}
If $(n_k)_{k\geqslant 1}$ satisfies \eqref{lac1} and \eqref{lac2}, then the sequences 
$(N^{-1/2}S_N^*(h,\cdot))_{N\geqslant 1}$ and 
$(N^{-1/2}S_N^{**}(h,\cdot))_{N\geqslant 1}$ are not tight in their respective 
spaces.
\end{cor}

Let $\delta>0$. Then the choice $n_k:= \lfloor 16^{(2+\delta)^k}\rfloor$ satisfies the 
conditions \eqref{lac1} and \eqref{lac2}.

Under assumptions of Proposition 7 (the choice of $n_k$) we get 
\ref{control_mixing} in \ref{A} and because 
\eqref{lac1}, \eqref{lac2} are satisfied, we get \ref{nwip} in \ref{A}.
 
 By Remark~\ref{remark_mixing}, we can construct in 
Proposition~\ref{mixing_subsequence} the sequence  
$(n_k)_{k\geqslant 1}$ in such a way that it also satisfies \eqref{lac1} and 
\eqref{lac2}; this yields \ref{nwip} in Theorem~\ref{A'}, and of course from 
Proposition~\ref{mixing_subsequence} itself, property d') in 
Theorem~\ref{A'} also holds.

\subsection{Proof of \ref{clt} and \ref{SD}}

Let us denote by $\sigma_N^2$ the variance of $S_N(h)$, that is, $\mathbb E[S_N(h)^2]$.

\begin{propo}\label{sigma}
Under the conditions \eqref{lac1} and \eqref{lac2}, we have $\sigma_N^2\lesssim N$. 
\end{propo}

\begin{proof}
 From \eqref{powers} and \eqref{lac1}, we deduce that 
\begin{equation}\label{variance1}
 \sum_{j=1}^{i(N)}\mathbb E[S_N(h_j)^2]\lesssim \sum_{j=1}^{i(N)}n_j\leqslant 2n_{i(N)}\lesssim N.
\end{equation}
Recall that $h_k=(I-U^{-n_k})s_k$ with $s_k:=\sum_{i=0}^{n_k-1}U^{-i}e_k$. Therefore, when $n_k\geqslant N$, 
we have by a similar computation as for \eqref{expr_sum2},

\begin{equation}
 S_N(h_k)=(I-U^{-n_k})\sum_{j=1}^{N-1}j(U^{j-n_k}+U^{N-j})e_k+N(I-U^{-n_k})
 \sum_{j=N}^{n_k}U^{j-n_k}e_k.
\end{equation}

The first term has a variance of order $N^3/n_k^2$,  and the variance of the 
second term is (at most) of order $N^2n_k/n_k^2$. We thus have that for 
$n_k\geqslant N$, $\mathbb E[S_N(h_k)^2]\lesssim N^2/n_k$, 
hence by \eqref{lac2},
\begin{align*}
 \sum_{k\geqslant i(N)+1}\mathbb E[S_N(h_k)^2]&\lesssim  \sum_{k\geqslant i(N)+1}\frac{N^2}{n_k} 
\nonumber\\& =\frac{N^2}{n_{i(N)+1}}+ \sum_{k\geqslant i(N)+2}\frac{N^2}{n_k}
\leqslant \frac{N^2}{n_{i(N)+1}}\left(1+\sum_{j\geqslant i(N)+2}\frac 1{j^2}\right)\nonumber,\\
\end{align*}
therefore,
\begin{equation}
\label{variance2}
 \sum_{k\geqslant i(N)+1}\mathbb E[S_N(h_k)^2]\lesssim N.
\end{equation}
Combining \eqref{variance1} and \eqref{variance2}, and using (for a fixed $N$) 
the independence of the sequence 
$(S_N(h_j))_{j\geqslant 1}$, we conclude that $
\sigma_N^2(h)=\sigma_N^2(g-g\circ T)\lesssim N$. 

When we add a mean-zero nondegenerate independent sequence $(m\circ 
T^i)_{i\in\Z}$ with moments of any order 
greater than $2$, the 
variance of the $N$th partial sum of $((m+h)\circ T^i)_{i\geqslant 1}$ is 
bounded above and below by a quantity 
proportional to $N$, hence \ref{SD} is satisfied in Theorems~\ref{A} 
and~\ref{A'}. By the observation in the 
paragraph "About the method of proof", \ref{clt} holds.
\end{proof}

\subsection{Moments of the coboundary and the transfer function}

One can wonder to which $\el^p$ space can $g$ and $g-g\circ 
T$ belong.
\begin{propo}\label{moments_of_h}
Under the conditions \eqref{lac1} and \eqref{lac2}, we have $g\in \el^p$ for $0<p<1$ and $g-g\circ T\in \el^p$ for 
each $p>0$.
\end{propo}
\begin{proof}
Let $g_k:=U^{-1}v_k$, where $\ds v_k=\sum_{j=0}^{n_k-1}U^{-j}s_k$ and $\ds s_k=\sum_{j=0}^{n_k-1}U^{-j}e_k$. Recall 
that $g=-\sum_{k=1}^\infty g_k$.
For $0<p<1$ and any two non-negative real numbers $a$ and $b$, we have $(a+b)^p\leqslant a^p+b^p$. This gives, 
using \eqref{expr_sum2},
\begin{align*}
\mathbb E\abs{g_k}^p&\leqslant \left(\sum_{j=1}^{n_k}j^p \mathbb E\abs{U^{-j}e_k}+
\sum_{j=1}^{n_k-1}(n_k-j)^p \mathbb E\abs{U^{-j+n_k}e_k}\right)\\
&=\frac 1{n_k^2}\left(n_k^p+2\sum_{j=1}^{n_k-1}j^p\right)\\
&\leqslant\frac{1}{n_k^2}\left(n_k^p+2n_k^{p+1}\right)\\
&\leqslant 3n_k^{p-1}.
\end{align*}
By \eqref{lac2}, we have $n_k\geqslant k!\cdot n_1$ hence the series $\sum_{k\geqslant 1}\mathbb 
E\abs{g_k}^p$ is convergent. This proves that $g\in\el^p$ for $0<p<1$.

Corollary 2.4. in \cite{MR2792580} states the following: given positive integers
 $t$ and $p$, $X_1,\dots,X_t$ 
independent random variables such that $\mu\ens{0\leqslant X_j\leqslant 1}=1$ 
for each $j\in \croch t$, then
\begin{equation}\label{bound_moments}
\mathbb E\left(X\right)^p\leqslant B_p\cdot\max\ens{\mathbb E(X),
(\mathbb E(X))^p},
\end{equation}
where $B_p$ is the $p$-th Bell's number (defined by the recursion relation 
$B_{p+1}=\sum_{k=0}^p\binom pk B_k$ and 
$B_0=B_1=1$) and 
$X:=\sum_{j=1}^tX_j$.

We shall show that the series $\sum_{k\geqslant 1}\norm{h_k}_{p}$ is convergent 
for any integer $p$. Fix 
$k\geqslant 1$, and let $t:=2n_k$, $X_j:=\abs{U^{j-2n_k}e_k}$.
Applying \eqref{bound_moments}, we get
\begin{equation*}
\norm{h_k}_{p}^p\leqslant B_p\cdot\max\ens{2n_k^{-1},(2n_k^{-1})^p}=
2B_p\cdot n_k^{-1} ,
\end{equation*}
hence $\norm{h_k}_{p}\leqslant (2B_p)^{1/p}\cdot n_k^{-1/p}$ and 
condition~\eqref{lac2} guarantees the convergence of the series 
$\sum_kn_k^{-1/p} $. One could also use Rosenthal's 
inequality.
\end{proof}
Since the added process has moments of any order, Proposition~\ref{moments_of_h} proves \ref{moments} in Theorems~\ref{A} and~\ref{A'}.
\begin{propo}\label{moment_transfer_function}
 The transfer function $g$ does not belong to $\el^1$.
\end{propo}
\begin{proof}
 Fix an integer $k$, and define for $1\leqslant j\leqslant n_k$:
\begin{equation*}
 E_j:=\ens{\abs{U^{-j}e_k}=1}\cap\bigcap_{i\in\ens{1,\dots,2n_k-1}
 \setminus\ens j}
 \ens{U^{-i}e_k=0}\cap \bigcap_{l\neq k}\ens{g_l=0}.
\end{equation*}
Since these sets are pairwise disjoint, $g=\sum_{k\geqslant 
1}g_k$, with $g_k:=U^{-1}\sum_{i=0}^{n_k-1}U^{-i}
\left[\sum_{h=0}^{n_k-1}U^{-h}e_k\right]$ and $g_l(\omega)=0$ if $l\neq k$ and 
$\omega\in 
\bigcup_{j=1}^{n_k}E_j$, we have the equality of functions 
\begin{align*}
\abs{g_k}\cdot\chi\left(\bigcup_{j=1}^{n_k}E_j\right)
&=\sum_{j=1}^{n_k}\abs{g_k}\cdot\chi(E_j)\\
&=\sum_{j=1}^{n_k}\chi(E_j)\cdot\abs{U^{-1}\sum_{i=0}^{n_k-1}U^{-i}
\left[\sum_{h=0}^{n_k-1}U^{-h}e_k\right]}\\
&=\sum_{j=1}^{n_k}\chi(E_j)\cdot\abs{\sum_{i=0}^{n_k-1}\sum_{h=0}^{n_k-1}
U^{-1-(i+h)}e_k}=\sum_{j=1}^{n_k}\chi(E_j)
\abs{j U^{-j}e_k}\\
&=\sum_{j=1}^{n_k}j\cdot \chi(E_j)
\end{align*}
and hence 
\begin{equation}\label{below_bound_g_k}
 \norm{\abs{g_k}\cdot \chi\left(\bigcup_{j=1}^{n_k}E_j\right)}_1=
 \sum_{j=1}^{n_k}j\cdot \mu(E_j).
\end{equation}

 As $\left(1-1/n_k^2\right)^{-1}\to 1$ for $k\to +\infty$ and 
 $\prod_{j\geqslant 1}\left(1-1/n_j^2\right)^{2n_j}$ is positive, we get 
\begin{equation}\label{below_bound_g_k_2}
 \mu(E_j)\geqslant\frac 1{n_k^2}\left(1-\frac 1{n_k^2}\right)^{2n_k-1}
 \prod_{l\neq k}\left(1-\frac 
1{n_l^2}\right)^{2n_l}\geqslant \frac c{n_k^2}
\end{equation}
for some positive constant $c$ independent of $k$ and $j$. 

Let us define $F_k:=\bigcup_{j=1}^{n_k}E_j$ for $k\geqslant 1$. 
Notice that the event $F_k$, $k\geqslant 1$ are pairwise disjoint because 
$F_k\subset\ens{g_k\neq 0}\cap \bigcap_{l\neq k}\ens{g_l=0}$. 
Therefore, combining \eqref{below_bound_g_k} and 
\eqref{below_bound_g_k_2}, we obtain for each $k\geqslant 1$, 
\[\mathbb E[|g|\chi(F_k)]=\mathbb E[|g_k|\chi(F_k)]\geqslant
c/2.\]

It then follows that \[\mathbb E\abs g\geqslant 
\sum_k\mathbb E[|g_k|\chi(F_k)]=\infty,\] proving 
Proposition~\ref{moment_transfer_function}.

\end{proof}

\textbf{Acknowledgements.} The authors would like to thank an anonymous referee
for a great number of helpful remarks and
corrections. In particular, the referee suggested the present proof of 
Lemma~\ref{mix} which is simpler
and easier to read than the proof in the original version. During main 
part of the research the second author was visiting the Department of 
Statistics of the University of Michigan; he thanks for the 
hospitality of U-M. Both authors thank Professor Magda Peligrad 
for suggesting the topic and for encouragement.

 \end{document}